\documentclass[a4paper,12pt]{article}

\usepackage[utf8]{inputenc}
\usepackage[english]{babel}
\usepackage{amsmath, amssymb, amsthm}
\usepackage{enumerate}
\usepackage[colorlinks=true,linkcolor=blue,citecolor=blue]{hyperref}

\usepackage{tikz}
\usetikzlibrary{fadings}
\tikzfading[name=fade out,
            inner color=transparent!0,
            outer color=transparent!100]
\usetikzlibrary{calc,arrows,decorations.pathreplacing,fadings,3d,positioning}

\title{On transference principle and Nesterenko's linear independence criterion}
\author{Oleg\,N.\,German, Nikolay\,G.\,Moshchevitin}
\date{}

\theoremstyle{definition}
\newtheorem{definition}{Definition}

\newtheorem*{notation*}{Notation}

\theoremstyle{remark}
\newtheorem{remark}{Remark}
\newtheorem*{remark*}{Remark}

\theoremstyle{plain}
\newtheorem{theorem}{Theorem}
\newtheorem{lemma}{Lemma}

\newtheorem{proposition}{Proposition}
\newtheorem{corollary}{Corollary}

\newtheorem*{question*}{Question}
\newtheorem*{statement*}{Statement}
\newtheorem*{corollary*}{Corollary}

\renewcommand{\phi}{\varphi}

\renewcommand{\vec}[1]{\mathbf{#1}}
\renewcommand{\geq}{\geqslant}
\renewcommand{\leq}{\leqslant}

\newcommand{\e}{\varepsilon}

\newcommand{\R}{\mathbb{R}}
\newcommand{\Z}{\mathbb{Z}}
\newcommand{\Q}{\mathbb{Q}}
\newcommand{\N}{\mathbb{N}}

\newcommand{\cC}{\mathcal{C}}

\newcommand{\cL}{\mathcal{L}}

\newcommand{\Glin}{G_{\textup{lin}}}
\newcommand{\Gsim}{G_{\textup{sim}}}

\begin{document}

\maketitle

\hfill
\textit{Dedicated to Yu.\,V.\,Nesterenko}


\hfill
\textit{on the occasion of his 75-th anniversary}

\vskip 8mm

\begin{abstract}
  \noindent
  We consider the problem of simultaneous approximation of real numbers $\theta_1, \ldots,\theta_n$ with rationals and the dual problem of approximating zero with the values of the linear form $x_0+\theta_1x_1+\ldots+\theta_nx_n$ at integer points. In this setting we analyse two transference inequalities obtained by Schmidt and Summerer. We present a rather simple geometric observation, which proves their result. We also derive several corollaries previously unknown. Particularly, we show that, together with the transference inequalities for uniform exponents, Schmidt and Summerer's inequalities imply the inequalities by Bugeaud and Laurent and ``one half'' of the inequalities by Marnat and Moshchevitin. Besides that, we show that our main construction provides a rather simple proof of Nesterenko's linear independence criterion.
\end{abstract}


\section{Introduction}\label{sec:intro}

This paper originated as a result of studying the prominent paper \cite{nesterenko_criterion} by Nesterenko, where he proves his linear independence criterion.

Given an integer $n\geq2$, let us fix an arbitrary $n$-tuple $\pmb\theta=(\theta_1,\ldots,\theta_n)\in\R^n$. The problem of approximating $\theta_1,\ldots,\theta_n$ simultaneously with rational numbers having equal denominators is well known to be related to the problem of approximating zero with the values of the linear form $x_0+\theta_1x_1+\ldots+\theta_nx_n$ at nonzero integer points. This relation is performed by the so called \emph{transference principle} discovered by Khintchine \cite{khintchine_palermo}. He formulated it in terms of \emph{Diophantine exponents}.

\begin{definition}\label{def:exp_sim}
  The \emph{(regular) Diophantine exponent} $\lambda=\lambda(\pmb\theta)$ is defined as the supremum of real numbers $\gamma$ such that the system of inequalities
  \begin{equation}\label{eq:exp_sim}
    \max_{1\leq i\leq n}|x_0\theta_i-x_i|\leq t^{-\gamma},\qquad
    0<|x_0|\leq t
  \end{equation}
  admits solutions in $\vec x=(x_0,\ldots,x_n)\in\Z^{n+1}$ for some arbitrarily large values of $t$.

  The respective \emph{uniform Diophantine exponent} $\hat\lambda=\hat\lambda(\pmb\theta)$ is defined as the supremum of real numbers $\gamma$ such that \eqref{eq:exp_sim} admits solutions in $\vec x=(x_0,\ldots,x_n)\in\Z^{n+1}$ for every $t$ large enough.
\end{definition}

\begin{definition}\label{def:exp_lin}
  The \emph{(regular) Diophantine exponent} $\omega=\omega(\pmb\theta)$ is defined as the supremum of real numbers $\gamma$ such that the system of inequalities
  \begin{equation}\label{eq:exp_lin}
    |x_0+\theta_1x_1+\ldots+\theta_nx_n|\leq t^{-\gamma},\qquad
    0<\max_{1\leq i\leq n}|x_i|\leq t
  \end{equation}
  admits solutions in $\vec x=(x_0,\ldots,x_n)\in\Z^{n+1}$ for some arbitrarily large values of $t$.

  The respective \emph{uniform Diophantine exponent} $\hat\omega=\hat\omega(\pmb\theta)$ is defined as the supremum of real numbers $\gamma$ such that \eqref{eq:exp_lin} admits solutions in $\vec x=(x_0,\ldots,x_n)\in\Z^{n+1}$ for every $t$ large enough.
\end{definition}


It follows immediately from Dirichlet's approximation theorem (or from Minkowski's convex body theorem) that the Diophantine exponents satisfy the trivial relations
$\lambda\geq\hat\lambda\geq1/n$ and $\omega\geq\hat\omega\geq n$.
It is also known (see \cite{jarnik_tiflis}) that, unless all the $\theta_i$ are rational, we have a slightly less trivial inequality
$\hat\lambda\leq1$.

Let us give a brief account on the existing nontrivial relations.

The aforementioned Khintchine's transference principle was published in 1926. It can be formulated as follows:
\begin{equation}\label{eq:khintchine}
  \frac{1+\omega}{1+\lambda}\geq n,\qquad
  \frac{1+\omega^{-1}}{1+\lambda^{-1}}\geq \frac1n\,.
\end{equation}
Later on, in 2007--10, Bugeaud and Laurent \cite{bugeaud_laurent_2007, bugeaud_laurent_2010} improved upon \eqref{eq:khintchine} by showing that
\begin{equation}\label{eq:bugeaud_laurent}
  \frac{1+\omega}{1+\lambda}\geq\frac{n-1}{1-\hat\lambda}\,,\qquad
  \frac{1+\omega^{-1}}{1+\lambda^{-1}}\geq \frac{1-\hat\omega^{-1}}{n-1}\,,
\end{equation}
provided that $1,\theta_1,\ldots,\theta_n$ are linearly independent over $\Q$. It can be easily verified that \eqref{eq:bugeaud_laurent} implies \eqref{eq:khintchine}, as $\hat\omega^{-1}\leq1/n\leq\hat\lambda\leq1$.

As for the uniform exponents, Jarn\'{\i}k \cite{jarnik_tiflis} proved in 1938 that, if $n=2$ and $1,\theta_1,\theta_2$ are linearly independent over $\Q$, a remarkable identity holds:
\begin{equation}\label{eq:jarnik_identity}
  \hat\omega^{-1}+\hat\lambda=1.
\end{equation}
In 2012 it was shown by German \cite{german_AA_2012, german_MJCNT_2012} that for arbitrary $n\geq2$ we have
\begin{equation}\label{eq:german_uniform_transference}
  \hat\omega\geq\frac{n-1}{1-\hat\lambda}\,,\qquad
  \hat\lambda\geq \frac{1-\hat\omega^{-1}}{n-1}\,.
\end{equation}
Clearly, \eqref{eq:german_uniform_transference} turns into Jarn\'{\i}k's identity for $n=2$.

In 2013 Schmidt and Summerer \cite{schmidt_summerer_2013} showed that
\begin{equation}\label{eq:german_moshchevitin}
  \hat\omega\leq\frac{1+\omega}{1+\lambda}\,,\qquad
  \hat\lambda\leq\frac{1+\omega^{-1}}{1+\lambda^{-1}}\,,
\end{equation}
provided that $1,\theta_1,\ldots,\theta_n$ are linearly independent over $\Q$. An alternative proof of their result can be found in \cite{german_moshchevitin_2013}. Clearly, \eqref{eq:german_uniform_transference} and \eqref{eq:german_moshchevitin} imply Bugeaud and Laurent's inequalities \eqref{eq:bugeaud_laurent}.

Furthermore, in 1950--54 Jarn\'{\i}k \cite{jarnik_szeged_1950, jarnik_czech_1954} proved for $n=2$ the inequalities
\begin{equation}\label{eq:jarnik_inequalities}
  \frac{\omega}{\hat\omega}\geq\hat\omega-1,\qquad
  \frac{\lambda}{\hat\lambda}\geq\frac{\hat\lambda}{1-\hat\lambda}\,.
\end{equation}
It is interesting to note that by \eqref{eq:jarnik_identity} the right-hand sides in \eqref{eq:jarnik_inequalities} coincide (for $n=2$) and are both equal to $\hat\omega\hat\lambda$, as well as to $(1-\hat\omega^{-1})/(1-\hat\lambda)$. We shall mention this fact in Section \ref{sec:corollaries_marnat_moshchevitin}.

Inequalities \eqref{eq:jarnik_inequalities} were generalised recently to the case of arbitrary $n\geq2$ by Marnat and Moshchevitin \cite{marnat_moshchevitin}. They showed that
\begin{equation}\label{eq:marnat_moshchevitin}
  \frac{\omega}{\hat\omega}\geq\Glin(\hat\omega),\qquad
  \frac{\lambda}{\hat\lambda}\geq\Gsim(\hat\lambda),
\end{equation}
where $\Glin(\hat\omega)$ and $\Gsim(\hat\lambda)$ are the largest roots of the polynomials
\begin{equation}\label{eq:marnat_moshchevitin_f_g}
  f(x)=\hat\omega^{-1}x^n-x+(1-\hat\omega^{-1}),
  \qquad
  g(x)=(1-\hat\lambda)x^n-x^{n-1}+\hat\lambda
\end{equation}
respectively. An alternative proof of the second inequality \eqref{eq:marnat_moshchevitin} can be found in \cite{nguyen_poels_roy}.

Summing up, we can say that all the nontrivial relations known up to now follow from \eqref{eq:german_uniform_transference}, \eqref{eq:german_moshchevitin}, and \eqref{eq:marnat_moshchevitin}. The main purpose of the current paper is to present a geometric observation, a rather simple one, which proves \eqref{eq:german_moshchevitin} almost immediately. It appears that this observation also provides a quite simple proof of Nesterenko's linear independence criterion.

\paragraph{}

The rest of the paper is organised as follows. In Section \ref{sec:empty_cylinder} we formulate and prove our main result and apply it to prove \eqref{eq:german_moshchevitin}. In Section \ref{sec:corollaries} we derive some corollaries to \eqref{eq:german_moshchevitin} concerning lower bounds for the ratios $\omega/\hat\omega$, $\lambda/\hat\lambda$ and show that \eqref{eq:german_moshchevitin} implies the weakest of inequalities \eqref{eq:marnat_moshchevitin}. Besides that, we analyse how \eqref{eq:german_uniform_transference} and \eqref{eq:german_moshchevitin} split Bugeaud and Laurent's inequalities \eqref{eq:bugeaud_laurent}, and compare some of our corollaries with a recent result by Schleischitz. Finally, Section \ref{sec:nesterenko} is devoted to Nesterenko's linear independence criterion. We present a rather simple proof of his theorem, which is based on our main geometric observation described in Section \ref{sec:empty_cylinder}. We also show that, in order to prove the linear independence criterion itself, it suffices to use the first step of Nesterenko's induction.

\section{Empty cylinder lemma}\label{sec:empty_cylinder}

Let us introduce some notation. Let us denote by $\ell$ the one-dimensional subspace generated by $(1,\theta_1,\ldots,\theta_n)$ and by $\ell^\perp$ its orthogonal complement. For every $\vec x\in\R^{n+1}$ let us denote by $r(\vec x)$ the Euclidean distance from $\vec x$ to $\ell$ and by $h(\vec x)$ the Euclidean distance from $\vec x$ to $\ell^\perp$. Let us also denote by $\langle\,\cdot\,,\cdot\,\rangle$ the inner product in $\R^{n+1}$.

Our main geometric observation is described by the following statement.

\begin{lemma}\label{l:empty_cylinder}
  Let $t,\alpha,\beta$ be positive real numbers such that $t^\beta>2t^\alpha$. Suppose $\vec v\in\Z^{n+1}$ satisfies
  \begin{equation}\label{eq:empty_cylinder_v}
    r(\vec v)=t^{\alpha-1-\beta},\qquad
    h(\vec v)=t^\alpha.
  \end{equation}
  Consider the half open cylinder
  \begin{equation}\label{eq:empty_cylinder_C}
    \cC=\cC(t,\alpha,\beta)=\Big\{\, \vec x\in\R^{n+1}\ \Big|\ r(\vec x)<t,\ t^{-\beta}\leq h(\vec x)\leq t^{-\alpha}-t^{-\beta} \,\Big\}.
  \end{equation}
  Then $\cC\cap\Z^{n+1}=\varnothing$.
\end{lemma}

\begin{proof}
  Consider an arbitrary $\vec x\in\cC$. Let us show that $0<\langle\vec v,\vec x\rangle<1$. To this end let us denote by $\vec y$ the orthogonal projection of $\vec x$ onto the two-dimensional subspace $\pi$ spanned by $\ell$ and $\vec v$. Then $r(\vec y)$ and $h(\vec y)$ can be interpreted as the absolute values of the coordinates of $\vec y$ in $\pi$ with respect to the coordinate axes $\ell$ and $\ell^\bot\cap\pi$ (see Figure \ref{fig:empty_cylinder}). Hence
  \[
    \langle\vec v,\vec x\rangle=
    \langle\vec v,\vec y\rangle\geq
    h(\vec v)h(\vec y)-r(\vec v)r(\vec y)>
    t^{\alpha-\beta}-t^{\alpha-\beta}=0
  \]
  and
  \[
    \langle\vec v,\vec x\rangle=
    \langle\vec v,\vec y\rangle\leq
    h(\vec v)h(\vec y)+r(\vec v)r(\vec y)<
    1-t^{\alpha-\beta}+t^{\alpha-\beta}=1.
  \]
  Thus, indeed, $0<\langle\vec v,\vec x\rangle<1$, and $\vec x$ cannot be an integer point.
\end{proof}

\begin{figure}[h]
\centering
\begin{tikzpicture}[scale=0.5]
    \draw[->,>=stealth'] (-12,0) -- (13.3,0) node[right] {$\ell^\perp\cap\pi$};
    \draw[->,>=stealth'] (0,-3) -- (0,14) node[above] {$\ell$};

    \draw (0,0) -- (3,12);
    \node[fill=black,circle,inner sep=1.2pt] at (3,12) {};
    \draw (3,12) node[right]{$\vec v$};

    \draw[color=black] plot[domain=-12:12] (\x, {-\x/4}) node[above right]{$\langle\vec v,\vec x\rangle=0$};
    \draw[color=black] plot[domain=-12:12] (\x, {-\x/4+7}) node[above right]{$\langle\vec v,\vec x\rangle=1$};

    \fill[blue,opacity=0.1] (8,2) -- (8,5) -- (-8,5) -- (-8,2) -- cycle;
    \draw (-8,2) -- (8,2);
    \draw (-8,5) -- (8,5);

    \draw[dashed] (-0.1,12) -- (3,12);
    \draw[dashed] (3,-0.1) -- (3,12);
    \draw[dashed] (8,-0.1) -- (8,2);
    \draw[dashed] (-8,-0.1) -- (-8,2);

    \draw (8,0) node[above right]{$t$};
    \draw (-8,0) node[above left]{$-t$};
    \draw (3,0) node[above right]{$t^{\alpha-1-\beta}$};
    \draw (0,12) node[above left]{$t^{\alpha}$};
    \draw (0,5) node[above left]{$t^{-\alpha}-t^{-\beta}$};
    \draw (0,2) node[above left]{$t^{-\beta}$};

    \draw (-7.1,3.5) node[right]{$\cC$};
\end{tikzpicture}
\caption{Empty cylinder} \label{fig:empty_cylinder}
\end{figure}
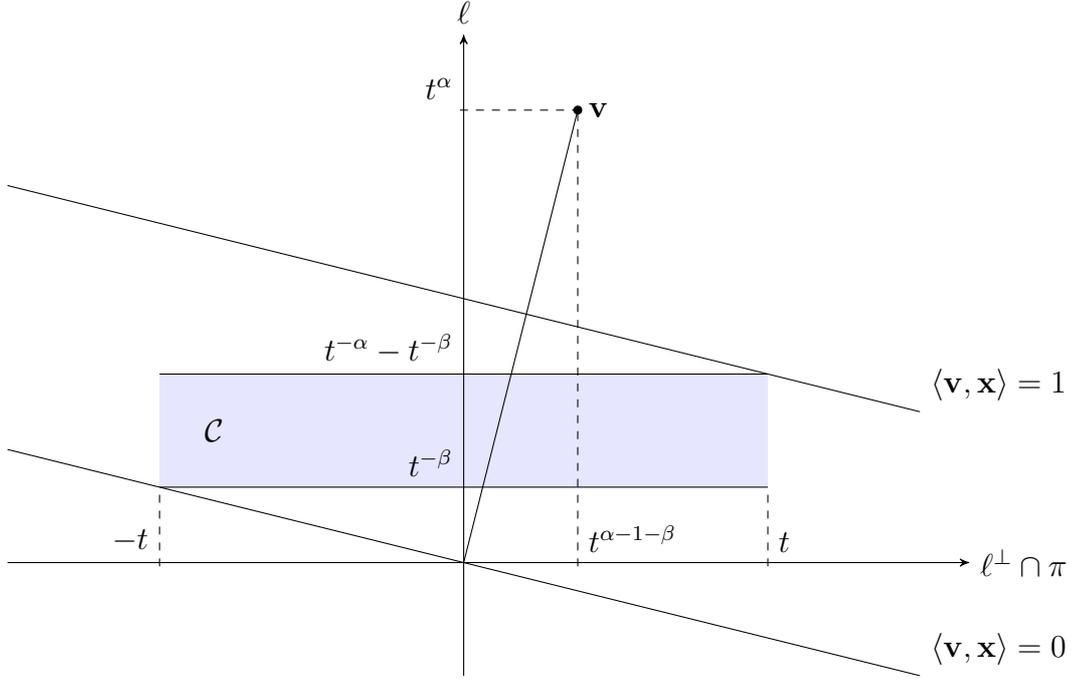

Now we can prove Schmidt and Summerer's inequalities \eqref{eq:german_moshchevitin}.

\begin{proof}[Proof of \eqref{eq:german_moshchevitin}]
  We may assume that $\lambda>1/n$ and $\omega>n$. Indeed, by Khintchine's inequalities \eqref{eq:khintchine} the equalities $\lambda=1/n$ and $\omega=n$ are equivalent, and if they hold, due to the trivial inequalities we also have $\hat\lambda=1/n$, $\hat\omega=n$. With such values of the exponents both inequalities \eqref{eq:german_moshchevitin} are obviously true. So, let us suppose that $\lambda>1/n$ and $\omega>n$.

  By the definition of $\lambda$ and $\omega$ we can choose a point $\vec v\in\Z^{n+1}$ and a positive real number $\gamma$ to satisfy any of the following two collections of conditions:

  \begin{enumerate}
    \item $h(\vec v)$ is arbitrarily large, $r(\vec v)=h(\vec v)^{-\gamma}$, $\gamma>1/n$, $\gamma$ is arbitrarily close to $\lambda$;
    \item $h(\vec v)$ is arbitrarily small, $r(\vec v)=h(\vec v)^{-\gamma}$, $\gamma<1/n$, $\gamma$ is arbitrarily close to $1/\omega$.
  \end{enumerate}

  We shall refer to the first choice as \emph{Case 1}, and to the second one as \emph{Case 2}.

  In either Case the main part of the argument is basically the same. So, let us choose $\vec v$ and $\gamma$ according to one of the two Cases. Set $t=t(\vec v,\gamma)$ to be the smallest positive real number such that the cylinder
  \begin{equation}\label{eq:C_v_definition}
    \cC_{\vec v}=\Big\{\, \vec x\in\R^{n+1}\ \Big|\ r(\vec x)\leq t,\ h(\vec x)\leq t\cdot h(\vec v)^{-1-\gamma} \,\Big\}
  \end{equation}
  contains a nonzero integer point. Define also $\alpha=\alpha(\vec v,\gamma)$ and $\beta=\beta(\vec v,\gamma)$ by the relations
  \begin{equation}\label{eq:alpha_beta_definition}
    h(\vec v)=t^\alpha,\qquad
    \alpha=\frac{1+\beta}{1+\gamma}\,.
  \end{equation}
  Then $r(\vec v)=h(\vec v)^{-\gamma}=t^{-\alpha\gamma}=t^{\alpha-1-\beta}$ and $t\cdot h(\vec v)^{-1-\gamma}=t^{-\beta}$. Hence $\vec v$ satisfies \eqref{eq:empty_cylinder_v} and
  \[
   \cC_{\vec v}=\Big\{\, \vec x\in\R^{n+1}\ \Big|\ r(\vec x)\leq t,\ h(\vec x)\leq t^{-\beta} \,\Big\}.
  \]
  In order to fulfil the hypothesis of Lemma \ref{l:empty_cylinder} it remains to provide that $t^\beta>2t^\alpha$. We remind that in Case 1 $h(\vec v)$ is arbitrarily large and $\gamma>1/n$, and in Case 2 $h(\vec v)$ is arbitrarily small and $\gamma<1/n$. In both Cases by Minkowski's convex body theorem the volume of $\cC_{\vec v}$ is bounded (by $2^{n+1}$), so the product $t^{n+1}h(\vec v)^{-1-\gamma}=h(\vec v)^{(n+1)/\alpha-(1+\gamma)}$ is also bounded. Hence in Case 1 we have $\alpha\gamma\geq(n+1)\gamma/(1+\gamma)>1$, and therefore, $\beta=\alpha+\alpha\gamma-1>\alpha$, whereas in Case 2 we have $\alpha\gamma\leq(n+1)\gamma/(1+\gamma)<1$, and therefore, $\beta=\alpha+\alpha\gamma-1<\alpha$. Thus, in both Cases we have
  \[
    t^\beta=
    h(\vec v)^{\beta/\alpha}>
    2h(\vec v)=
    2t^\alpha.
  \]
  Having the hypothesis of Lemma \ref{l:empty_cylinder} fulfilled, we conclude that there are no integer points in the cylinder $\cC=\cC(t,\alpha,\beta)$ defined by \eqref{eq:empty_cylinder_C}. At the same time there are some nonzero integer points in $\cC_{\vec v}$, but all of them lie on the boundary of $\cC_{\vec v}$. Moreover, every such point satisfies the condition $r(\vec x)=t$, since
  \[
    \Big\{\, \vec x\in\cC_{\vec v}\ \Big|\ r(\vec x)<t,\ h(\vec x)=t^{-\beta} \,\Big\}=
    \cC_{\vec v}\cap(\cC\cup(-\cC))
  \]
  and $\cC\cup(-\cC)$ is empty. Thus, there is an integer point satisfying
  \[
    r(\vec x)=t,\qquad
    h(\vec x)\leq t^{-\beta},
  \]
  and there are no nonzero integer points satisfying
  \[
    r(\vec x)<t,\qquad
    h(\vec x)\leq t^{-\alpha}-t^{-\beta}.
  \]

  Suppose now that Case 1 holds. Then $\beta>\alpha$ and $t$ can be arbitrarily large. Hence $\omega\geq\beta$, $\hat\omega\leq\alpha$, and
  \[
    \hat\omega\leq\alpha=
    \frac{1+\beta}{1+\gamma}\leq
    \frac{1+\omega}{1+\lambda-\e}
  \]
  with $\e$ positive and arbitrarily small, which gives us the first inequality \eqref{eq:german_moshchevitin}.

  Suppose that Case 2 holds. Then $\beta<\alpha$ and $t$ can be arbitrarily small. Hence $\lambda\geq1/\beta$, $\hat\lambda\leq1/\alpha$, and
  \[
    \hat\lambda\leq\alpha^{-1}=
    \frac{1+\gamma}{1+\beta}\leq
    \frac{1+\omega^{-1}+\e}{1+\lambda^{-1}}
  \]
  with $\e$ positive and arbitrarily small, which gives us the second inequality \eqref{eq:german_moshchevitin}.
\end{proof}

\section{Corollaries to Schmidt and Summerer's inequalities}\label{sec:corollaries}

\subsection{Lower bounds for the ratios $\omega/\hat\omega$ and $\lambda/\hat\lambda$}\label{sec:corollaries_marnat_moshchevitin}

Inequalities \eqref{eq:german_moshchevitin} can be rewritten as
\begin{equation}\label{eq:german_moshchevitin_rewritten}
  \frac{\omega}{\hat\omega}\geq\frac{1+\lambda}{1+\omega^{-1}}\,,\qquad
  \frac{\lambda}{\hat\lambda}\geq\frac{1+\lambda}{1+\omega^{-1}}\,,
\end{equation}
giving thus new lower bounds in the problem of estimating the ratio between the regular and the uniform exponents.

Note that, same as in Jarn\'{\i}k's inequalities \eqref{eq:jarnik_inequalities} for $n=2$, the right-hand sides in \eqref{eq:german_moshchevitin_rewritten} coincide. However, the lower estimate for $\omega/\hat\omega$ and $\lambda/\hat\lambda$ provided by \eqref{eq:german_moshchevitin_rewritten} is a function of $\omega$ and $\lambda$, whereas it has been more traditional to think in this context of functions of $\hat\omega$ and $\hat\lambda$. It appears that the right-hand side $(1+\lambda)/(1+\omega^{-1})$ can be weakened to $(1-\hat\omega^{-1})/(1-\hat\lambda)$, generalising thus Jarn\'{\i}k's inequalities \eqref{eq:jarnik_inequalities} to arbitrary dimension.

\begin{proposition}\label{prop:B_is_not_less_than_A}
  Suppose $1,\theta_1,\ldots,\theta_n$ are linearly independent over $\Q$. Then
  \begin{equation}\label{eq:B_is_not_less_than_A}
    \frac{1+\lambda}{1+\omega^{-1}}\geq
    \frac{1-\hat\omega^{-1}}{1-\hat\lambda}\,.
  \end{equation}
\end{proposition}

\begin{proof}
  Set $B=(1+\lambda)/(1+\omega^{-1})$. Then \eqref{eq:german_moshchevitin_rewritten} reads as $\omega\geq\hat\omega B$ and $\lambda\geq\hat\lambda B$. Applying \eqref{eq:german_moshchevitin_rewritten} to $\omega$ and $\lambda$ in $B$, we get
  \[
    B\geq\frac{1+\hat\lambda B}{1+\hat\omega^{-1} B^{-1}}\,,
  \]
  or, equivalently,
  \[
    B\geq\frac{1-\hat\omega^{-1}}{1-\hat\lambda}\,.
  \]
\end{proof}

\begin{corollary}\label{cor:german_moshchevitin_weakened_to_A}
  Suppose $1,\theta_1,\ldots,\theta_n$ are linearly independent over $\Q$. Then
  \begin{equation}\label{eq:german_moshchevitin_weakened_to_A}
    \frac{\omega}{\hat\omega}\geq
    \frac{1-\hat\omega^{-1}}{1-\hat\lambda}\,,\qquad
    \frac{\lambda}{\hat\lambda}\geq
    \frac{1-\hat\omega^{-1}}{1-\hat\lambda}\,.
  \end{equation}
\end{corollary}

Inequalities \eqref{eq:german_moshchevitin_weakened_to_A} coincide with Jarn\'{\i}k's inequalities \eqref{eq:jarnik_inequalities} for $n=2$. Another generalisation of \eqref{eq:jarnik_inequalities} to arbitrary dimension is provided by \eqref{eq:marnat_moshchevitin}. An important difference between \eqref{eq:marnat_moshchevitin} and \eqref{eq:german_moshchevitin_weakened_to_A} is that the lower bounds $\Glin(\hat\omega)$ and $\Gsim(\hat\lambda)$ in \eqref{eq:marnat_moshchevitin} are functions of only one exponent, whereas \eqref{eq:german_moshchevitin_weakened_to_A} requires both.


It appears that if $\Glin(\hat\omega)=\Gsim(\hat\lambda)$, then \eqref{eq:marnat_moshchevitin} is equivalent to \eqref{eq:german_moshchevitin_weakened_to_A}, but if $\Glin(\hat\omega)\neq\Gsim(\hat\lambda)$, then \eqref{eq:german_moshchevitin_weakened_to_A} is stronger than the weakest of \eqref{eq:marnat_moshchevitin} and weaker than the strongest.

\begin{proposition}\label{prop:the_alternative}
  Suppose $1,\theta_1,\ldots,\theta_n$ are linearly independent over $\Q$. Then either
  \begin{equation}\label{eq:the_alternative_1}
    \Glin(\hat\omega)\leq
    (\hat\omega\hat\lambda)^{\frac1{n-1}}\leq
    \frac{1-\hat\omega^{-1}}{1-\hat\lambda}\leq
    \Gsim(\hat\lambda),
  \end{equation}
  or
  \begin{equation}\label{eq:the_alternative_2}
    \Gsim(\hat\lambda)\leq
    \frac{1-\hat\omega^{-1}}{1-\hat\lambda}\leq
    (\hat\omega\hat\lambda)^{\frac1{n-1}}\leq
    \Glin(\hat\omega).
  \end{equation}
  Moreover, if any two of the four quantities under comparison coincide, then so do all of them.
\end{proposition}

\begin{proof}
  Let $f(x)$ and $g(x)$ be the polynomials defined by \eqref{eq:marnat_moshchevitin_f_g}. Then, given $t\geq1$, we have
  \[
    t>\Glin(\hat\omega)\iff f(t)>0,\qquad
    t>\Gsim(\hat\lambda)\iff g(t)>0.
  \]
  Since $\hat\omega\geq n$ and $\hat\lambda\geq1/n$, both $(\hat\omega\hat\lambda)^{\frac1{n-1}}\geq1$ and $(1-\hat\omega^{-1})/(1-\hat\lambda)\geq1$. By simple substitution and regrouping it is easily verified that
  \[
    g\bigg(\frac{1-\hat\omega^{-1}}{1-\hat\lambda}\bigg)\geq0\iff
    \frac{1-\hat\omega^{-1}}{1-\hat\lambda}\leq
    (\hat\omega\hat\lambda)^{\frac1{n-1}}\iff
    f\Big((\hat\omega\hat\lambda)^{\frac1{n-1}}\Big)\leq0
  \]
  and that the respective equalities are equivalent. This proves the statement.
\end{proof}

\begin{corollary}\label{cor:german_marnat_moshchevitin}
  Denote $B(\omega,\lambda)=\dfrac{1+\lambda}{1+\omega^{-1}}$\,, $A(\hat\omega,\hat\lambda)=\dfrac{1-\hat\omega^{-1}}{1-\hat\lambda}$\,. Then, assuming that $1,\theta_1,\ldots,\theta_n$ are linearly independent over $\Q$, we have
  \begin{equation}\label{eq:german_marnat_moshchevitin}
    \frac{\omega}{\hat\omega}\geq
    \max\big(B(\omega,\lambda),\Glin(\hat\omega)\big),\qquad
    \frac{\lambda}{\hat\lambda}\geq
    \max\big(B(\omega,\lambda),\Gsim(\hat\lambda)\big).
  \end{equation}
  Moreover, $B(\omega,\lambda)\geq A(\hat\omega,\hat\lambda)\geq\min\big(\Glin(\hat\omega),\Gsim(\hat\lambda)\big)$.
\end{corollary}

\begin{remark}
  In case $n=2$ all the four quantities in Proposition \ref{prop:the_alternative} coincide giving rise to the relation
  \[
    \hat\omega\hat\lambda=
    \frac{1-\hat\omega^{-1}}{1-\hat\lambda}\,,
  \]
  which is the very Jarn\'{\i}k's identity \eqref{eq:jarnik_identity}. Thus, Proposition \ref{prop:the_alternative} provides a kind of generalisation of this identity to arbitrary dimension in the following form:
  \[
    (\hat\omega\hat\lambda)^{\frac1{n-1}}=
    \frac{1-\hat\omega^{-1}}{1-\hat\lambda}\,.
  \]
  However, this identity holds if and only if $\Glin(\hat\omega)=\Gsim(\hat\lambda)$, which happens quite rarely. Nevertheless, if a strict inequality holds, then the sign of this inequality determines uniquely which of the quantities $\Glin(\hat\omega)$, $\Gsim(\hat\lambda)$ is larger.
\end{remark}

\subsection{Inequalities by Bugeaud and Laurent}

Inequalities \eqref{eq:bugeaud_laurent} obviously follow from \eqref{eq:german_moshchevitin} and \eqref{eq:german_uniform_transference}:
\begin{equation}\label{eq:loranoyadenie_plus}
  \frac{1+\omega}{1+\lambda}\geq\hat\omega\geq\frac{n-1}{1-\hat\lambda}\,,\qquad
  \frac{1+\omega^{-1}}{1+\lambda^{-1}}\geq\hat\lambda\geq \frac{1-\hat\omega^{-1}}{n-1}\,.
\end{equation}
Bugeaud and Laurent proved \eqref{eq:bugeaud_laurent} with the help of so called \emph{intermediate Diophantine exponents} by splitting Khintchine's inequalities \eqref{eq:khintchine} into ``going up'' and ``going down'' chains of inequalities between consecutive intermediate exponents and improving on the first inequality in each chain with the respective uniform exponent. An analogous splitting of German's inequalities \eqref{eq:german_uniform_transference} can be found in \cite{german_AA_2012}. It would be interesting to establish a suitable splitting of our inequalities \eqref{eq:german_moshchevitin}. Such a splitting, combined with the existing splitting of \eqref{eq:german_uniform_transference}, would provide a splitting of \eqref{eq:bugeaud_laurent}, alternative to the one of Bugeaud and Laurent.

As for \eqref{eq:loranoyadenie_plus} itself, the way it splits \eqref{eq:bugeaud_laurent} with the help of the respective fourth exponent gives additional information, for instance, in the case when one of relations \eqref{eq:bugeaud_laurent} is an equality. Then the respective pair of inequalities \eqref{eq:loranoyadenie_plus} becomes a pair of equalities and we get a whole range of equalities for the intermediate exponents mentioned above. It is worth mentioning in this context an explicit description of triples of exponents, for which either the first pair of inequalities \eqref{eq:loranoyadenie_plus}, or the second one, is a pair of equalities, obtained recently by Schleischitz \cite{schleischitz_laureaud_nonsharpness_2021}.

\subsection{Inequalities by Schleischitz}

Analysing the quantity $(\hat\omega\hat\lambda)^{\frac1{n-1}}$ appearing in Proposition \ref{prop:the_alternative} provides the following observation.


\begin{proposition}\label{prop:schleischitz_dual_uniform}
  Suppose $1,\theta_1,\ldots,\theta_n$ are linearly independent over $\Q$. Then
  \begin{equation}\label{eq:schleischitz_dual}
    \hat\lambda\leq
    \frac{\omega^{n-1}}{\hat\omega^n}\,,\qquad
    \lambda\leq
    \frac{\omega^n}{\hat\omega^{n+1}}\,.
  \end{equation}
\end{proposition}

\begin{proof}
  Applying \eqref{eq:german_marnat_moshchevitin} and \eqref{eq:B_is_not_less_than_A} in both cases \eqref{eq:the_alternative_1} and \eqref{eq:the_alternative_2}, we get
  \[
    \frac{\omega}{\hat\omega}\geq
    \max\bigg(\dfrac{1+\lambda}{1+\omega^{-1}}\,,\,\Glin(\hat\omega)\bigg)\geq
    \max\bigg(\dfrac{1-\hat\omega^{-1}}{1-\hat\lambda}\,,\,\Glin(\hat\omega)\bigg)\geq
    (\hat\omega\hat\lambda)^{\frac1{n-1}}.
  \]
  Hence the first inequality \eqref{eq:schleischitz_dual} follows immediately. 
  
  Furthermore\footnote{This part of the proof was proposed by Johannes Schleischitz}, by the first inequality \eqref{eq:german_moshchevitin} we have
  \[
    \lambda\leq
    \frac{\omega-\hat\omega+1}{\hat\omega}=
    \frac{\omega^n}{\hat\omega^{n+1}}-f\big(\omega/\hat\omega\big),
  \]
  where $f(x)$ is as in \eqref{eq:marnat_moshchevitin_f_g}. Since $f\big(\omega/\hat\omega\big)\geq0$ by \eqref{eq:marnat_moshchevitin}, we get the second inequality \eqref{eq:schleischitz_dual}.
\end{proof}

In \cite{schleischitz_hungarica_2021} Schleischitz proves two inequalities that very much resemble \eqref{eq:schleischitz_dual}. They have the form
\begin{equation}\label{eq:schleischitz_rough}
  \hat\omega\leq
  \frac{\lambda^{n-1}}{\hat\lambda^n}+\psi(\lambda,\hat\lambda),\qquad
  \omega\leq
  \frac{\lambda^n}{\hat\lambda^{n+1}}+\chi(\lambda,\hat\lambda)
\end{equation}
with some nonnegative $\psi(\lambda,\hat\lambda)$ and $\chi(\lambda,\hat\lambda)$ turning into zero only if $\lambda/\hat\lambda=\Gsim(\hat\lambda)$. One might wonder whether \eqref{eq:schleischitz_rough} can be relieved from $\psi(\lambda,\hat\lambda)$ or $\chi(\lambda,\hat\lambda)$, so that it would become similar to \eqref{eq:schleischitz_dual}. However, the answer is negative, for, as it follows from results of Kleinbock, Moshchevitin, and Weiss \cite{kleinbock_moshchevitin_weiss} there exist many $\pmb\theta$ with small $\lambda$ and infinite $\hat\omega$. Their argument can be modified to show that, preserving $\lambda$ small, we can make $\hat\omega$ finite, but however large. Thus, it is impossible to substitute the extra summands with zero in the general case.

\section{Nesterenko's linear independence criterion}\label{sec:nesterenko}

In 1985 Nesterenko published his famous linear independence criterion. His original proof was rather involved. A simpler argument can be found in \cite{fischler_zudilin} and \cite{chantanasiri}, see also \cite{fischler_rivoal_2010}. It appears that our Lemma \ref{l:empty_cylinder} provides an even simpler proof.

As before, let us fix $\pmb\theta=(\theta_1,\ldots,\theta_n)\in\R^n$. Let $\ell$, $\ell^\perp$, $r(\,\cdot\,)$, $h(\,\cdot\,)$ be defined as in the beginning of Section \ref{sec:empty_cylinder}.

For each subspace $\cL$ of $\R^{n+1}$ let us denote by $\varphi(\cL)$ the tangent of the angle between $\ell$ and $\cL$. If $\cL$ is defined over $\Q$, we denote by $H(\cL)$ its height, i.e. the covolume of the lattice $\cL\cap\Z^{n+1}$.

\begin{theorem}[Nesterenko, 1985]\label{t:nesterenko}
  Let $\alpha,\beta,c_1,c_2,\e$ be positive real numbers, $\beta\geq\alpha$.
  Let $(t_k)_{k\in\N}$ be an increasing sequence of positive real numbers such that
  \begin{equation}\label{eq:nesterenko_t_k}
    \lim_{k\to\infty}t_k=\infty,\qquad
    \limsup_{k\to\infty}\frac{\log(t_{k+1})}{\log(t_k)}=1.
  \end{equation}
  Suppose that for every integer $k$ large enough there exists $\vec x\in\Z^{n+1}$ such that
  \begin{equation}\label{eq:nesterenko_cylinder}
    r(\vec x)\leq t_k,\qquad
    c_1t_k^{-\beta}\leq h(\vec x)\leq c_2t_k^{-\alpha}.
  \end{equation}
  Then, for every $d$-dimensional subspace $\cL$ of $\R^{n+1}$ defined over $\Q$, there is a positive $c_3=c_3(d,\e)$ such that
  \begin{equation}\label{eq:nesterenko_angle}
    \varphi(\cL)\geq c_3H(\cL)^{-\delta-\e},\qquad
    \delta=\delta(d)=\frac{1+\beta}{1+\beta-d(1+\beta-\alpha)}\,,
  \end{equation}
  provided $d<(1+\beta)/(1+\beta-\alpha)$.
\end{theorem}


\subsection{Simple proof of Nesterenko's theorem}\label{sec:simple_proof}

Let us derive Theorem \ref{t:nesterenko} from Lemma \ref{l:empty_cylinder}. We split our argument into three steps. Suppose that the hypothesis of Theorem \ref{t:nesterenko} holds.

\paragraph{Step 1: Adjusting the hypothesis.}

Let $\e'$ be a positive real number sufficiently smaller than $\e$.
Set $\alpha'=(1-\e')\alpha$, $\beta'=(1+\e')\beta$.
Let us show that for every real $t$ large enough there exists $\vec x\in\Z^{n+1}$ such that
\begin{equation}\label{eq:nesterenko_cylinder_germanised}
  r(\vec x)<t,\qquad
  t^{-\beta'}<h(\vec x)<t^{-\alpha'}-t^{-\beta'}.
\end{equation}
For every large $t$ choose $k$ so that $t_k<t\leq t_{k+1}$. It follows from \eqref{eq:nesterenko_t_k} that $t_{k+1}<t_k^{1+\e'/2}$ if $k$ is large enough. Hence for $\vec x\in\Z^{n+1}$ satisfying \eqref{eq:nesterenko_cylinder} we have $r(\vec x)\leq t_k<t$ and
\[
  t^{-\beta'}<
  c_1t_k^{-\beta}\leq
  h(\vec x)\leq
  c_2t_k^{-\alpha}<
  c_2t^{-\alpha(1-\e'/2)}<
  t^{-\alpha'}-t^{-\beta'},
\]
if $k$ and $t$ are large enough. Thus, \eqref{eq:nesterenko_cylinder_germanised} is valid.

\paragraph{Step 2: Proof for $d=1$.}

Let $\cL$ be a one-dimensional subspace generated by a primitive $\vec v\in\Z^{n+1}$ (primitive means that the coordinates of $\vec v$ are coprime). Then $\varphi(\cL)=r(\vec v)/h(\vec v)$ and $H(\cL)=\sqrt{r(\vec v)^2+h(\vec v)^2}$. Let us suppose that $h(\vec v)$ is large and that
\[
  r(\vec v)<h(\vec v)^{1-\delta'},\qquad
  \delta'=\frac{1+\beta'}{\alpha'}\,.
\]
Set $t=h(\vec v)^{1/\alpha'}$. Then $h(\vec v)=t^{\alpha'}$ and $r(\vec v)=t^{\alpha'-1-\beta''}$ with some $\beta''>\beta'$. By Lemma \ref{l:empty_cylinder} there are no integer points $\vec x$ such that
\[
  r(\vec x)<t,\qquad
  t^{-\beta''}\leq h(\vec x)\leq t^{-\alpha'}-t^{-\beta''},
\]
which contradicts the existence of integer $\vec x$ satisfying \eqref{eq:nesterenko_cylinder_germanised} if $t$ is large enough. Thus, if $h(\vec v)$ is large enough, we have
\begin{equation}\label{eq:nesterenko_angle_dimension_1}
  r(\vec v)\geq h(\vec v)^{1-\delta'},
\end{equation}
which implies \eqref{eq:nesterenko_angle} for $d=1$, as $\e'$ can be arbitrarily small.

\paragraph{Step 3: Proof for $d\geq2$.}

Suppose $2\leq d<(1+\beta')/(1+\beta'-\alpha')$. Then, particularly,
\begin{equation}\label{eq:bounds_for_alpha_beta_delta}
  \beta'>\alpha'>1,\qquad
  1<\delta'<\frac{d}{d-1}\,.
\end{equation}
Given a $d$-dimensional subspace $\cL$ defined over $\Q$, let $\ell'$ be the one-dimensional subspace generated by the point in $\cL$ closest to the point $(1,\theta_1,\ldots,\theta_n)$. Then the angle between $\ell$ and $\cL$ equals the angle between $\ell$ and $\ell'$. Furthermore, for each $\vec x\in\cL$ let us denote by $r'(\vec x)$ and $h'(\vec x)$ the Euclidean distances from $\vec x$ to $\ell'$ and to $\cL\cap\ell^\perp$ respectively.

By Minkowski's convex body theorem, for each positive $t$, there is a nonzero integer point $\vec x$ in $\cL$ such that
\begin{equation}\label{eq:bounds_for_v_before_v}
  h'(\vec x)\leq t,\qquad
  r'(\vec x)\leq c_d\big(H(\cL)\big/t\big)^{1/(d-1)}
\end{equation}
with some positive $c_d$ depending only on $d$. Set
\begin{equation}\label{eq:t_0}
  t_0=\big((4c_d)^{d-1}H(\cL)\big)^{1/(d-(d-1)\delta')}.
\end{equation}
Choose a nonzero integer point $\vec v$ in $\cL$ (see Figure \ref{fig:step_3}), so that, in accordance with \eqref{eq:bounds_for_v_before_v},
\begin{equation}\label{eq:bounds_for_v}
  h'(\vec v)\leq t_0,\qquad
  r'(\vec v)\leq c_d\big(H(\cL)\big/t_0\big)^{1/(d-1)}.
\end{equation}
It follows from \eqref{eq:bounds_for_alpha_beta_delta} that $t_0/H(\cL)\to\infty$ and $r'(\vec v)\to0$ as $H(\cL)\to\infty$. We may also assume that $\varphi(\cL)\to0$ as $H(\cL)\to\infty$, neglecting most of subspaces. Then $h'(\vec v)\to\infty$ as $H(\cL)\to\infty$ and
\begin{equation}\label{eq:h_vs_h'}
  h(\vec v)\leq
  h'(\vec v)<
  2h(\vec v).
\end{equation}
Having thus chosen $\vec v$ and assuming $H(\cL)$ to be large enough, we apply \eqref{eq:nesterenko_angle_dimension_1}, \eqref{eq:bounds_for_alpha_beta_delta}, \eqref{eq:t_0}, \eqref{eq:bounds_for_v}, \eqref{eq:h_vs_h'} and estimate $\varphi(\cL)$ as follows:
\begin{multline*}
  \varphi(\cL)\geq
  \frac{r(\vec v)}{h(\vec v)}-\frac{r'(\vec v)}{h(\vec v)}>
  h(\vec v)^{-\delta'}-2\frac{r'(\vec v)}{h'(\vec v)}\geq
  h'(\vec v)^{-\delta'}-\frac{2c_d(H(\cL))^{1/(d-1)}}{t_0^{1/(d-1)}h'(\vec v)}\geq \\
  \geq
  h'(\vec v)^{-\delta'}\bigg(1-\frac{2c_d(H(\cL))^{1/(d-1)}}{t_0^{1/(d-1)}h'(\vec v)^{1-\delta'}}\bigg)\geq \\
  \geq
  t_0^{-\delta'}\bigg(1-\frac{2c_d(H(\cL))^{1/(d-1)}}{t_0^{-\delta'+d/(d-1)}}\bigg)=
  \frac{t_0^{-\delta'}}2=
  \frac12\big((4c_d)^{d-1}H(\cL)\big)^{-\delta'/(d-(d-1)\delta')}\,.
\end{multline*}
Taking into account that $\delta'/(d-(d-1)\delta')=(1+\beta')/(1+\beta'-d(1+\beta'-\alpha'))$, we get \eqref{eq:nesterenko_angle}, as $\e'$ can be arbitrarily small.

\begin{figure}[h]
  \centering
  \begin{tikzpicture}

  \begin{scope}[scale=1,x=0.8cm,y=1cm,z=-0.4cm]

  \coordinate (v) at (-2,4,2);
  \coordinate (v') at (-2,4,0);
  \coordinate (v'') at (0,4,0);
  \coordinate (v''') at (0,0,2);
  \coordinate (w) at (0,0,-5);

  \draw (w) arc[x radius=7.5cm, y radius=2.07cm, start angle=74.5, end angle=-240];
  \draw[rotate=20] (w) arc[x radius=2.6cm, y radius=8cm, start angle=9, end angle=195];
  \draw[color=white, fill=white] (0,5.65,0) circle [radius=2mm];
  \draw (0,0,-5) -- (0,0,5);

  \draw (0,0,0) -- (0,8,0) node[right] {$\ell$};
  \draw (0,0,0) -- (-4,8,0) node[right] {$\ell'$};
  \draw[color=white, fill=white] ($0.882*(v')$) circle [radius=1mm];

  \draw[dashed] (v) -- (v') -- (v'') -- cycle;
  \draw[dashed] (v) -- (v''');

  \draw (0,0,0) -- (v) node[below left] {$\vec v$};
  \node[fill=black,circle,inner sep=1.4pt] at (v) {};

  \node[above left] at ($0.48*(v)+0.48*(v')$) {$r'(\vec v)$};
  \node[below right] at ($0.523*(v)+0.523*(v'')$) {$r(\vec v)$};
  \node[below left] at ($0.45*(v)+0.45*(v''')$) {$h'(\vec v)$};
  \node[right] at ($0.5*(v'')$) {$h(\vec v)$};
  \node[right] at (4.5,0,-4.5) {$\ell^\perp$};
  \node[left] at (-4.7,5,0) {$\cL$};

  \draw ($(v)-(v')-(0,0,0.3)$) -- ($(v)-(v')-(0,0,0.3)+(-0.1,0.2,0)$) -- ($(v)-(v')+(-0.1,0.2,0)$);
  \draw ($(v')+(0,0,0.3)$) -- ($(v')+(0,0,0.3)+(-0.1,0.2,0)$) -- ($(v')+(-0.1,0.2,0)$);
  \draw ($(v'')+(-0.1,0,0.17)$) -- ($(v'')+(-0.1,0,0.17)+(0,0.2,0)$) -- ($(v'')+(0,0.2,0)$);

  \end{scope}

  \end{tikzpicture}
  \caption{Step 3}\label{fig:step_3}
\end{figure}

\subsection{Concerning the linear independence criterion itself}

Theorem \ref{t:nesterenko} implies the estimate
\begin{equation}\label{eq:nesterenko_estimate}
  \dim_{\Q}(\Q+\Q\theta_1+\ldots+\Q\theta_n)\geq
  \frac{1+\beta}{1+\beta-\alpha}\,.
\end{equation}
However, this estimate easily follows from \eqref{eq:nesterenko_angle} with $d=1$, i.e. we may actually avoid Step 3, if \eqref{eq:nesterenko_estimate} is our aim. Indeed, if $\ell$ is contained in a $d$-dimensional subspace $\cL$ defined over $\Q$, then Dirichlet's approximation theorem guarantees that there are infinitely many integer points $\vec x$ in $\cL$ such that
\[
  r(\vec x)\leq c(d,H(\cL))h(\vec x)^{-1/(d-1)}
\]
with some positive $c(d,H(\cL))$. But if $d<(1+\beta')/(1+\beta'-\alpha')$, we have $\delta'<d/(d-1)$, where $\alpha',\beta',\delta'$ are as in Section \ref{sec:simple_proof}, and Step 2 tells us that for every integer $\vec x$ with $h(\vec x)$ large enough we have the opposite:
\[
  r(\vec x)\geq
  h(\vec x)^{1-\delta'}>
  c(d,H(\cL))h(\vec x)^{-1/(d-1)}.
\]
Thus, there is no such $\cL$, and we get \eqref{eq:nesterenko_estimate}.

\subsection{A slight refinement in the case of linear independence}

Theorem \ref{t:nesterenko} is mainly applied to prove linear independence of numbers $1,\theta_1,\dots,\theta_n$ over $\Q$. But if it is already known that they are linearly independent, estimate \eqref{eq:nesterenko_angle} of Theorem \ref{t:nesterenko} can be improved for $d=n$. In this case the restriction on $\alpha$ and $\beta$ in Theorem \ref{t:nesterenko} implies that $\alpha>n-1$. For $\alpha>n$ the following statement provides a stronger inequality than \eqref{eq:nesterenko_angle}. It also provides a lower bound for $\omega=\omega(\pmb\theta)$.

\begin{proposition}\label{prop:nesterenko_slightly_improved}
  Within the hypothesis of Theorem \ref{t:nesterenko}, suppose additionally that $1,\theta_1,\dots,\theta_n$ are linearly independent over $\Q$ and that
  \begin{equation}\label{eq:nesterenko_angle_slightly_improved_denominator}
    \frac{(\alpha-1)(1+\beta)}{\alpha(1+\beta-\alpha)}>n-1.
  \end{equation}
  Let $\cL$ be an $n$-dimensional subspace of $\R^{n+1}$ defined over $\Q$. Then
  \begin{equation}\label{eq:nesterenko_angle_slightly_improved_dual}
    \omega\leq
    \frac{(n-1)\alpha(1+\beta-\alpha)}{(\alpha-1)(1+\beta)-(n-1)\alpha(1+\beta-\alpha)}
  \end{equation}
  and
  \begin{equation}\label{eq:nesterenko_angle_slightly_improved}
    \varphi(\cL)\geq cH(\cL)^{-\delta-\e},\qquad
    \delta=\frac{(\alpha-1)(1+\beta)}{(\alpha-1)(1+\beta)-(n-1)\alpha(1+\beta-\alpha)}\,,
  \end{equation}
  with some positive $c$ depending only on $n$ and $\e$.
\end{proposition}

\begin{proof}
  It follows from \eqref{eq:nesterenko_cylinder_germanised} and \eqref{eq:nesterenko_angle_dimension_1} of Steps 1 and 2 that
  \[
    \hat\omega\geq\alpha,\qquad
    \lambda\leq\frac{1+\beta}{\alpha}-1.
  \]
  By \eqref{eq:bugeaud_laurent} we also have
  \[
    \frac{1+\omega^{-1}}{1+\lambda^{-1}}\geq \frac{1-\hat\omega^{-1}}{n-1}\,.
  \]
  Hence, taking into account that \eqref{eq:nesterenko_angle_slightly_improved_denominator} implies positiveness of both denominators appearing further, we get
  \[
    \omega\leq
    \frac{n-1}{\big(1-\hat\omega^{-1}\big)\big(1+\lambda^{-1}\big)-(n-1)}\leq
    \frac{(n-1)\alpha(1+\beta-\alpha)}{(\alpha-1)(1+\beta)-(n-1)\alpha(1+\beta-\alpha)}\,.
  \]
  This proves \eqref{eq:nesterenko_angle_slightly_improved_dual}. 

  Furthermore, let $\vec w$ be the integer normal of $\cL$, i.e. the unique (up to sign) nonzero primitive integer vector orthogonal to $\cL$. Bounds for $\varphi(\cL)$ and $\omega$ are obviously related, as $H(\cL)=\sqrt{h(\vec w)^2+r(\vec w)^2}$ and the angle between $\cL$ and $\ell$ is equal to the angle between $\vec w$ and $\ell^\perp$. Hence
  \[
    \varphi(\cL)=
    h(\vec w)/r(\vec w)\geq
    cH(\cL)^{-\omega-1-\e}\geq
    cH(\cL)^{-\delta-\e}
  \]
  with some positive $c$ depending only on $n$ and $\e$. Thus, \eqref{eq:nesterenko_angle_slightly_improved} is also proved.
\end{proof}

\paragraph{Acknowledgements.}

The authors are grateful to Johannes Schleischitz for a series of useful comments, especially, for drawing the authors' attention to Schmidt and Summerer's paper \cite{schmidt_summerer_2013} once again.

The first author is a winner of the ``Junior Leader'' contest conducted by Theoretical Physics and Mathematics Advancement Foundation “BASIS” and would like to thank its sponsors and jury.


\end{document}